\documentclass[final,onefignum,onetabnum]{siamart171218}

\usepackage{braket,amsfonts}

\usepackage{array}

\usepackage[caption=false]{subfig}
\usepackage{pgfplots}

\newsiamthm{claim}{Claim}
\newsiamremark{remark}{Remark}
\newsiamremark{hypothesis}{Hypothesis}
\crefname{hypothesis}{Hypothesis}{Hypotheses}

\usepackage{algorithmic}

\usepackage{graphicx,epstopdf}
\usepackage[export]{adjustbox}

\Crefname{ALC@unique}{Line}{Lines}

\usepackage{amsopn}

\usepackage{xspace}
\usepackage{bold-extra}
\usepackage[most]{tcolorbox}

\colorlet{texcscolor}{blue!50!black}
\colorlet{texemcolor}{red!70!black}
\colorlet{texpreamble}{red!70!black}
\colorlet{codebackground}{black!25!white!25}

\lstdefinestyle{siamlatex}{%
  style=tcblatex,
  texcsstyle=*\color{texcscolor},
  texcsstyle=[2]\color{texemcolor},
  keywordstyle=[2]\color{texemcolor},
  moretexcs={cref,Cref,maketitle,mathcal,text,headers,email,url},
}

\tcbset{%
  colframe=black!75!white!75,
  coltitle=white,
  colback=codebackground, 
  colbacklower=white, 
  fonttitle=\bfseries,
  arc=0pt,outer arc=0pt,
  top=1pt,bottom=1pt,left=1mm,right=1mm,middle=1mm,boxsep=1mm,
  leftrule=0.3mm,rightrule=0.3mm,toprule=0.3mm,bottomrule=0.3mm,
  listing options={style=siamlatex}
}

\newtcblisting[use counter=example]{example}[2][]{%
  title={Example~\thetcbcounter: #2},#1}

\newtcbinputlisting[use counter=example]{\examplefile}[3][]{%
  title={Example~\thetcbcounter: #2},listing file={#3},#1}

\DeclareTotalTCBox{\code}{ v O{} }
{ 
  fontupper=\ttfamily\color{black},
  nobeforeafter,
  tcbox raise base,
  colback=codebackground,colframe=white,
  top=0pt,bottom=0pt,left=0mm,right=0mm,
  leftrule=0pt,rightrule=0pt,toprule=0mm,bottomrule=0mm,
  boxsep=0.5mm,
  #2}{#1}

\patchcmd\newpage{\vfil}{}{}{}
\begin{tcbverbatimwrite}{tmp_\jobname_header.tex}
\title{Limitations of Richardson Extrapolation for Kernel Density Estimation\thanks{Submitted to the editors 12/18/18.}}

\author{Ruben G. Ascoli\thanks{Thomas Jefferson High School for Science and Technology, Alexandria, VA  (\email{rubenascoli01@gmail.com}). Adviser: Tyrus Berry, Department of Applied Math, George Mason University, Fairfax, VA (\email{tberry@gmu.edu})}}

\headers{Richardson Extrapolation for Kernel Density Estimation}{Ruben G. Ascoli}
\end{tcbverbatimwrite}
\input{tmp_\jobname_header.tex}

\ifpdf
\hypersetup{ pdftitle={SIAM KDE Paper} }
\fi

\begin{document}
\maketitle

\begin{tcbverbatimwrite}{tmp_\jobname_abstract.tex}
\begin{abstract}
  This paper develops the process of using Richardson Extrapolation to improve the Kernel Density Estimation method, resulting in a more accurate (lower Mean Squared Error) estimate of a probability density function for a distribution of data in $\mathbb{R}^d$ given a set of data from the distribution. The method of Richardson Extrapolation is explained, showing how to fix conditioning issues that arise with higher-order extrapolations. Then, it is shown why higher-order estimators do not always provide the best estimate, and it is discussed how to choose the optimal order of the estimate. It is shown that given $n$ one-dimensional data points, it is possible to estimate the probability density function with a mean squared error value on the order of only $n^{-1}\sqrt{\ln(n)}$. Finally, this paper introduces a possible direction of future research that could further minimize the mean squared error.
\end{abstract}


\end{tcbverbatimwrite}
\input{tmp_\jobname_abstract.tex}

\section{Introduction}
\label{sec:intro}
The Kernel Density Estimator (KDE) is a non-parametric method of estimating the underlying probability density function of some unknown distribution given a sample of data from that distribution. The KDE always yields a smooth function, and the fact that it is non-parametric means the estimated probability density function can take any shape, whether bimodal, sinusoidal, normal, or any other shape \cite{ScottVBK}.

Research on Kernel Density Estimation has gained popularity since early papers on the subject by Rosenblatt (1956) \cite{Rosenblatt56}, Whittle (1958) \cite{whittle58}, and Parzen (1962) \cite{Parzen62}. This research has a variety of applications in data collection, especially as we consider higher-dimensional data and the limits of accuracy we can get to even with these more elusive probability density functions. Cacoullos (1966) first considered multivariate density estimation \cite{cacoullos66}.

The KDE of a sample $X_1, X_2, X_3, ..., X_n \in \mathbb{R}^d$ of data of size $n$ from an unknown distribution $P$ with twice-differentiable probability density function $p(x)$, using the Gaussian kernel, is given by \[\hat{p}_n(x) = \frac{1}{nh^d}\sum_{i=1}^n \frac{e^{-(||x-X_i||/h)^2/2}}{(2\pi)^{d/2}},\] where $\hat{p}_n(x)$ is the estimated probability density function at point $x$ and $h$ is the bandwidth parameter that determines the amount of smoothing of the data \cite{ScottVBK}. The value of $h$ is generally small. Note that various kernels can be used for KDE, but in this paper we only consider the Gaussian kernel. The KDE turns each data point into a smooth Gaussian bump and adds up all of these bumps to get a smooth estimated probability density function.

This estimate is approximately equal to the actual probability density function $p$ plus some error which decreases with larger values of $n$. This paper uses the mean squared error (MSE) value as an indicator of error. For one-dimensional data sets (i.e. $d=1$), Whittle showed that the MSE cannot decrease faster than $n^{-1}$ \cite{whittle58}, whereas we can easily make the error decrease like $n^{-4/5}$, as explained by Scott (1980) \cite{positiveHigherOrder} and as will be shown in \cref{sec:error}. For the general case with $d$-dimensional data, the error will instead not be able to decrease faster than $n^{-1/2}$ for large values of $d$, but we will show that a decrease proportional to $n^{-4/4+d}$ can be attained rather easily.

In this paper, we first show the calculations for the MSE in \cref{sec:error}. Then, in \cref{sec:rich}, we expand on this using Richardson Extrapolation to determine a better estimate that gives us a lower error. In \cref{sec:lagrange}, we show how to continue this method and surpass obstacles that rise up against Richardson Extrapolation. Finally, in \cref{sec:limitations}, we obtain our results and discuss the limitations of this approach, and in \cref{sec:results} and \cref{sec:future}, we summarize what we gain by using this method and explore a possible direction for future research.

 We now turn to a detailed derivation of the bias, variance, and expected mean squared error of the estimator which will be needed in \cref{sec:rich}.

\section{Bias, Variance, and Mean Squared Error}
\label{sec:error}
The kernel density estimator of a probability density function deviates from the actual underlying probability density function by a certain error. Our measure of error is the mean squared error, which we calculate in this section. These calculations are standard, as in Rosenblatt \cite{Rosenblatt56}, but they form the basis for later parts of this paper.
\subsection{Bias}
\label{sec:bias}
We begin with this statement regarding the bias of the kernel density estimator:
\begin{theorem}\label{thm:biasmag} Given a $(2k)$-times differentiable probability density function $p(x)$ with bounded derivatives of order $2k$, if the random data $X_1, ..., X_n$ sampled from $p(x)$ are independent and identically distributed (IID), the bias of the kernel density estimator for $p(x)$ is \[\mathbb{E}[\hat{p}_n(x)] - p(x) = m_2(x)h^2 + m_4(x)h^4+...+m_{2k}(x)h^{2k}.\] \end{theorem}
\begin{proof}
We start with this statement, coming directly from the definition of KDE: \[\mathbb{E}[\hat{p}_n(x)] = \mathbb{E}\left[\frac{1}{nh^d}\sum_{i=1}^n \frac{e^{-(||x-X_i||/h)^2/2}}{(2\pi)^{d/2}}\right].\] We assumed that the data $X_1, ..., X_n$ are independent and identically distributed, so that since this is an expectation value, we can replace $X_i$ with $X_1$ (or any single $X$ value) in this expression. \[\mathbb{E}[\hat{p}_n(x)] = \mathbb{E}\left[\frac{1}{nh^d}\sum_{i=1}^n \frac{e^{-(||x-X_1||/h)^2/2}}{(2\pi)^{d/2}}\right].\] 

Now, within the sum, there is no dependence on $i$, so we can replace the sum with $n$ times the addend. Then, the $n$ cancels with the $n$ in front of the sum, so we get \[\mathbb{E}[\hat{p}_n(x)] = \mathbb{E}\left[\frac{1}{h^d}\frac{e^{-(||x-X_1||/h)^2/2}}{(2\pi)^{d/2}}\right].\] Noting that $X_1$ is our variable here (rather than $x$ or $h$, which are fixed), we replace $X_1$ with $y$ in our expression to avoid confusion. We now turn this expectation value into an integral $dy$. \[\mathbb{E}[\hat{p}_n(x)] = \int_{\mathbb{R}^d} \frac{1}{h^d}\frac{e^{-(||x-y||/h)^2/2}}{(2\pi)^{d/2}}p(y)\ dy.\] 

In order to evaluate this integral, we use change of variables. We substitute \[u = \frac{y-x}{h}, \ \ du = det\left(\frac{du}{dy}\right)dy = \frac{1}{h^d}dy, \ \ y = uh+x.\] The limits of integration remain the same. Our equation becomes \[\mathbb{E}[\hat{p}_n(x)] = \int_{\mathbb{R}^d} \frac{e^{-u^2/2}}{(2\pi)^{d/2}}p(uh+x)\ du.\] If we choose a small value for $h$, then we can use Taylor's theorem to expand this integral. We assume the probability density function is ($2k$)-times differentiable. Using multi-index notation, with $\alpha=(\alpha_1, \alpha_2, ..., \alpha_d)$, $|\alpha| = \alpha_1+\alpha_2+...+\alpha_d$, and $\alpha! = \alpha_1!\alpha_2!...\alpha_d!$, Taylor's theorem says \[p(x+uh) = \sum_{|\alpha|\leq 2k-1}\frac{\partial^\alpha p(x)}{\alpha!} (uh)^\alpha + \sum_{|\alpha|=2k} \frac{\partial^\alpha p(x+cuh)}{\alpha!} (uh)^\alpha \mbox{ for some } c\in (0,1).\] 
The equation for the expectation value of $\hat{p}_n(x)$ becomes
\[\mathbb{E}[\hat{p}_n(x)] = \sum_{|\alpha|\leq 2k-1}\int_{\mathbb{R}^d} \frac{e^{-u^2/2}}{(2\pi)^{d/2}}\frac{\partial^\alpha p(x)}{\alpha!} (uh)^\alpha\ du.\]

We now recognize three things: firstly, assuming that the density function has bounded $(2k)$\textsuperscript{th} derivatives, all the integrals in this sum are finite due to the exponential decay factor; secondly, the $|\alpha|=0$ in this sum is simply $p(x)$ times the integral of the probability density function of the standard normal distribution, which gives $p(x)$ times $1$; lastly, all the integrals with odd $|\alpha|$ are 0 since we integrate an odd function of $u$ over a domain that is symmetric about the origin. Our equation becomes\[\mathbb{E}[\hat{p}_n(x)] - p(x) = m_2(x)h^2 + m_4(x)h^4+...+m_{2k}(x)h^{2k},\] where \[m_{2j}(x) = \sum_{|\alpha|=2j}\int_{\mathbb{R}^d} \frac{e^{-u^2/2}}{(2\pi)^{d/2}}\frac{\partial^\alpha p(x)}{\alpha!} u^\alpha\ du\] for $j\neq k$ and \[m_{2k}(x) = \sum_{|\alpha|=2k}\int_{\mathbb{R}^d} \frac{e^{-u^2/2}}{(2\pi)^{d/2}}\frac{\partial^\alpha p(x+cuh)}{\alpha!} u^\alpha\ du \mbox{ for some } c\in (0,1).\] 
\end{proof}

Note that, although in this proof we assumed the derivatives of order $2k$ to be bounded, the bias would be finite even if these derivatives had polynomial growth due to the exponential decay factor in the integrand.

The following corollary is more relevant for the base use of KDE.
\begin{corollary}
Given a twice-differentiable probability density function $p(x)$ with a bounded second derivative, if the random data $X_1, ..., X_n$ sampled from $p(x)$ are IID, the bias of the kernel density estimator for $p(x)$ has magnitude on the order of $h^2$; in other words, \[\mathbb{E}[\hat{p}_n(x)]-p(x)=O(h^2).\]
\end{corollary}

\begin{proof}
We start with the result of \cref{thm:biasmag}. Since $h$ is small, we only take the $h^2$ term and ignore the rest as higher order.
We must assume that the density function has a bounded second derivative so that $m_2(x)$ is finite. We call the maximum value that the second derivative could possibly achieve $l$. Therefore, the bias value is \textit{at most} \[\mathbb{E}[\hat{p}_n(x)]-p(x)=\frac{lh^2}{2}\] The bias is therefore on the order of $ O(h^2).$
\end{proof}

Note that, for now, we are dismissing the value of $l$ as a constant that does not change the order of the bias. This will not be true in later sections of this paper.

Although picking a small $h$ would give us that the expected value for the estimated probability density function is very close to the actual probability density function, we should not be so quick to simply pick a very small value of $h$ because this would result in a very large variance.

\subsection{Variance}
\label{sec:var}
The variance measures the amount that our actual estimate for the probability density function differs from the expected value of the estimated probability density function and can be measured by \[\sigma^2 = \mathbb{E}[(\hat{p}_n-\mathbb{E}[\hat{p}_n])^2].\] 

As will be shown, this value approaches infinity as $h$ approaches $0$, meaning that although the expected value of the estimated probability density function will approach the actual probability density function, our actual estimated probability density function might be all over the place, resulting in a high mean squared error despite the low bias.

\begin{theorem}\label{thm:varmag} Given a probability density function p(x), the variance of its kernel density estimator has magnitude on the order of $1/(nh^d)$; in other words, \[\sigma^2 = \mathbb{E}[(\hat{p}_n-\mathbb{E}[\hat{p}_n])^2]=O\left(\frac{1}{nh^d}\right).\] \end{theorem}
\begin{proof}

Plugging in the expression for $\hat{p}_n$ and rearranging terms, we get \[\sigma^2 = \frac{1}{n^2}\mathbb{E}\left[\left(\sum_{i=1}^n \left(\frac{e^{-((x-X_i)/h)^2/2}}{(2\pi)^{d/2}h^d}-\mathbb{E}[\hat{p}_n(x)]\right)\right)^2\right].\] We now use the fact that the variance of a sum is the sum of the variances. Also, as before, we use the fact that the data $X_i$ are IID so that we can replace $X_i$ with $X_1$ in this expression. Then, without any dependence on $i$ inside the sum, we replace the sum with $n$ times the addend. \[\sigma^2 = \frac{1}{n}\mathbb{E}\left[\left(\frac{e^{-((x-X_1)/h)^2/2}}{(2\pi)^{d/2}h^d}-\mathbb{E}[\hat{p}_n(x)]\right)^2\right].\] 

As before, we replace $X_1$ with $y$ to avoid confusion. This expectation value is now turned into an integral. Also, we are still considering a small value of $h$, so we can plug in the expression for $\mathbb{E}[\hat{p}_n(x)]$ that we obtained before. \[\sigma^2 = \frac{1}{n}\int_{\mathbb{R}^d} \left(\frac{e^{-((x-y)/h)^2/2}}{(2\pi)^{d/2}h^d}-p(x)+O(h^2)\right)^2p(y)\ dy.\] 

Again, as before, we will use the substitution \[u = \frac{y-x}{h}, \ \ du = det\left(\frac{du}{dy}\right)dy = \frac{1}{h^d}dy, \ \ y = uh+x.\] Our equation becomes \[\sigma^2 = \frac{h^d}{n}\int_{\mathbb{R}^d} \left(\frac{e^{-u^2/2}}{(2\pi)^{d/2}h^d}-p(x)+O(h^2)\right)^2p(uh+x)\ du.\] We will now expand this expression, noting that the $O(h^2)$ is insignificant compared to the other terms. We get \[\sigma^2 = \frac{h^d}{n}\int_{\mathbb{R}^d} \left(\frac{e^{-u^2}}{(2\pi)^d h^{2d}}-\frac{2}{(2\pi)^{d/2}}\frac{e^{-u^2/2}}{h^d}p(x)+p(x)^2\right)p(uh+x)\ du.\] We see that the first term inside this integral is much greater than the other terms as $h$ gets small, so we simplify this into \[\sigma^2 = \frac{h^d}{n}\int_{\mathbb{R}^d} \frac{e^{-u^2}}{(2\pi)^d h^{2d}}p(uh+x)\ du+H.O.T.,\] where $H.O.T.$ means Higher Order Terms. Using a similar analysis as when calculating bias, we can say that \[\sigma^2 = \frac{1}{nh^d}p(x)\int_{\mathbb{R}^d} \frac{e^{-u^2}}{(2\pi)^d}\ du+H.O.T.\] That integral simplifies to $1/(2\pi)^{d/2}$. Therefore, the variance is given by \[\sigma^2 = \frac{1}{nh^d(2\pi)^{d/2}}p(x)+H.O.T. = O\left(\frac{1}{nh^d}\right).\]

\end{proof}

Thus, as $h$ approaches $0$, although the expected value of the estimated density function approaches the actual density function as shown in \cref{thm:biasmag}, the variance approaches infinity. There is a trade off that must be considered when deciding on a value for $h$ - when we pick a value too large, we will not get a good estimate of the actual probability density function, but when we pick a value too small, the variance becomes too large. What we really want to do is minimize the mean squared error (MSE), which depends on both the bias and the variance.

\subsection{Mean Squared Error}
\label{sec:mse}
The mean squared error takes into account both the bias and the variance with the following equality:
\begin{theorem}\label{thm:msebv}
The mean squared error of the kernel density estimator is equal to its variance plus its bias squared. In other words, \[\mathbb{E}[(\hat{p}_n(x)-p(x))^2] = \sigma^2+bias^2.\]
\end{theorem}
\begin{proof}
This proof involves some algebraic manipulation, starting with the definition of mean squared error.
\begin{align*}
\mathbb{E}[(\hat{p}_n(x)-p(x))^2] 
&= \mathbb{E}[(\hat{p}_n(x)-\mathbb{E}[\hat{p}_n(x)]+\mathbb{E}[\hat{p}_n(x)]-p(x))^2]\\
&= \mathbb{E}[(\hat{p}_n(x)-\mathbb{E}[\hat{p}_n(x)])^2+2(\hat{p}_n(x)-\mathbb{E}[\hat{p}_n(x)])(\mathbb{E}[\hat{p}_n(x)]-p(x))\\ &\qquad+(\mathbb{E}[\hat{p}_n(x)]-p(x))^2]\\
&=  \mathbb{E}[(\hat{p}_n(x)-\mathbb{E}[\hat{p}_n(x)])^2+(\mathbb{E}[\hat{p}_n(x)]-p(x))^2]\\
&=\mathbb{E}[(\hat{p}_n(x)-\mathbb{E}[\hat{p}_n(x)])^2]+(\mathbb{E}[\hat{p}_n(x)]-p(x))^2\\
&= \sigma^2+bias^2,
\end{align*}
where we have added and subtracted the same term, $\mathbb{E}[\hat{p}_n(x)],$ in order to show that this is the value of the variance, $\sigma^2$, plus the bias squared.
\end{proof}

This MSE is what we actually want to minimize, and this is done when we set the variance and bias squared equal to each other. So up to a factor of a constant, we get $1/(nh^d) = h^4.$ In other words, we want to choose a value for $h$ that is a constant times $n^{-1/(d+4)}.$ We then get an error value on the order of  \[\sigma^2+bias^2=O(h^4)=O(n^{-4/(d+4)}).\] Note that when $d=1$ we get an MSE value on the order of $n^{-4/5}$, as promised by Scott \cite{positiveHigherOrder}.

\cref{fig:fig1} shows estimates of the standard normal distribution in one and two dimensions using the Kernel Density Estimation. In the first graph, the estimated and actual functions are compared, while in the second graph, only the estimate is shown.

\begin{figure}[tbhp]\centering\subfloat[ ]{\label{fig:1a}\includegraphics[scale=0.35]{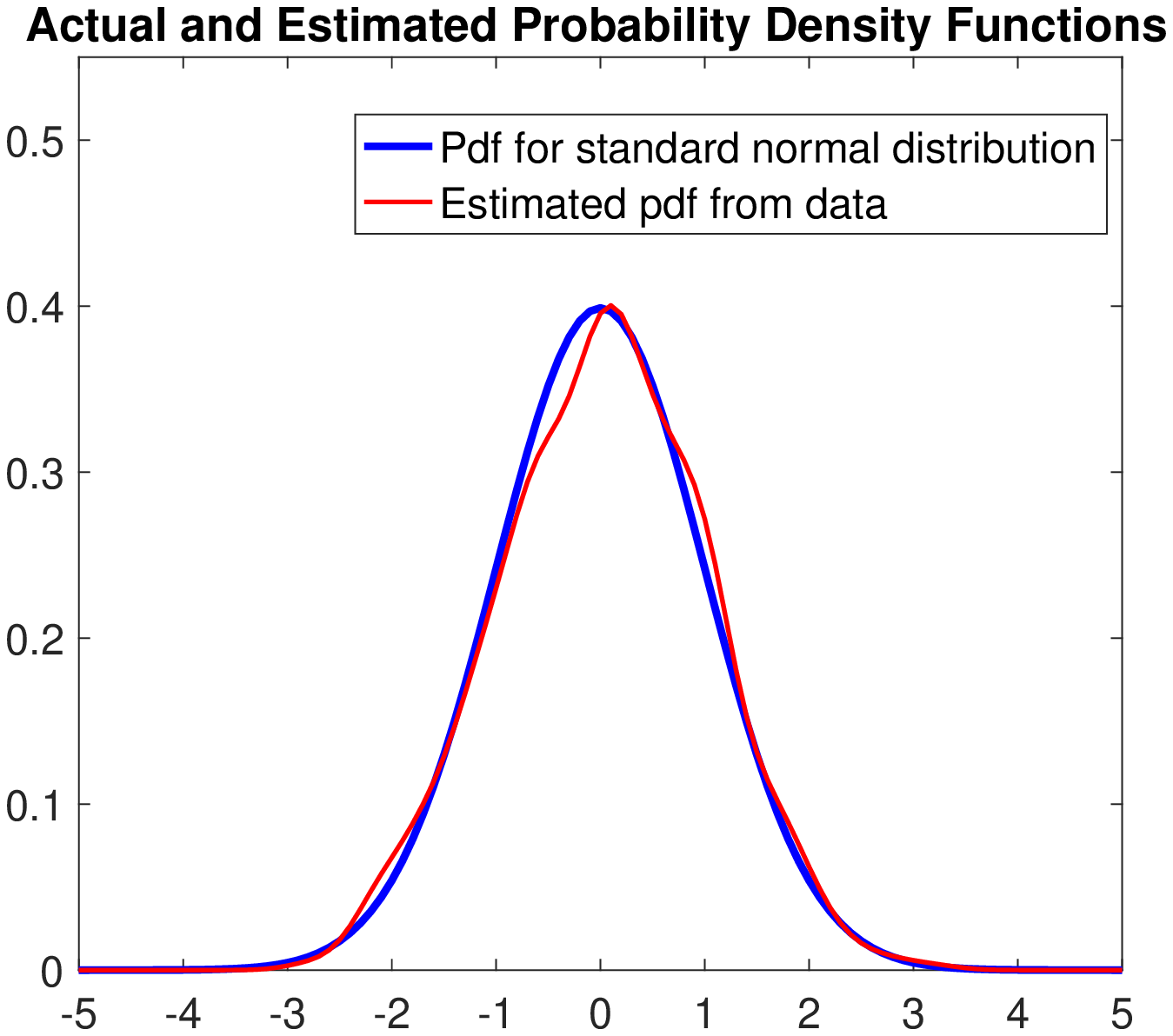}}\subfloat[ ]{\label{fig:1b}\includegraphics[scale=0.34]{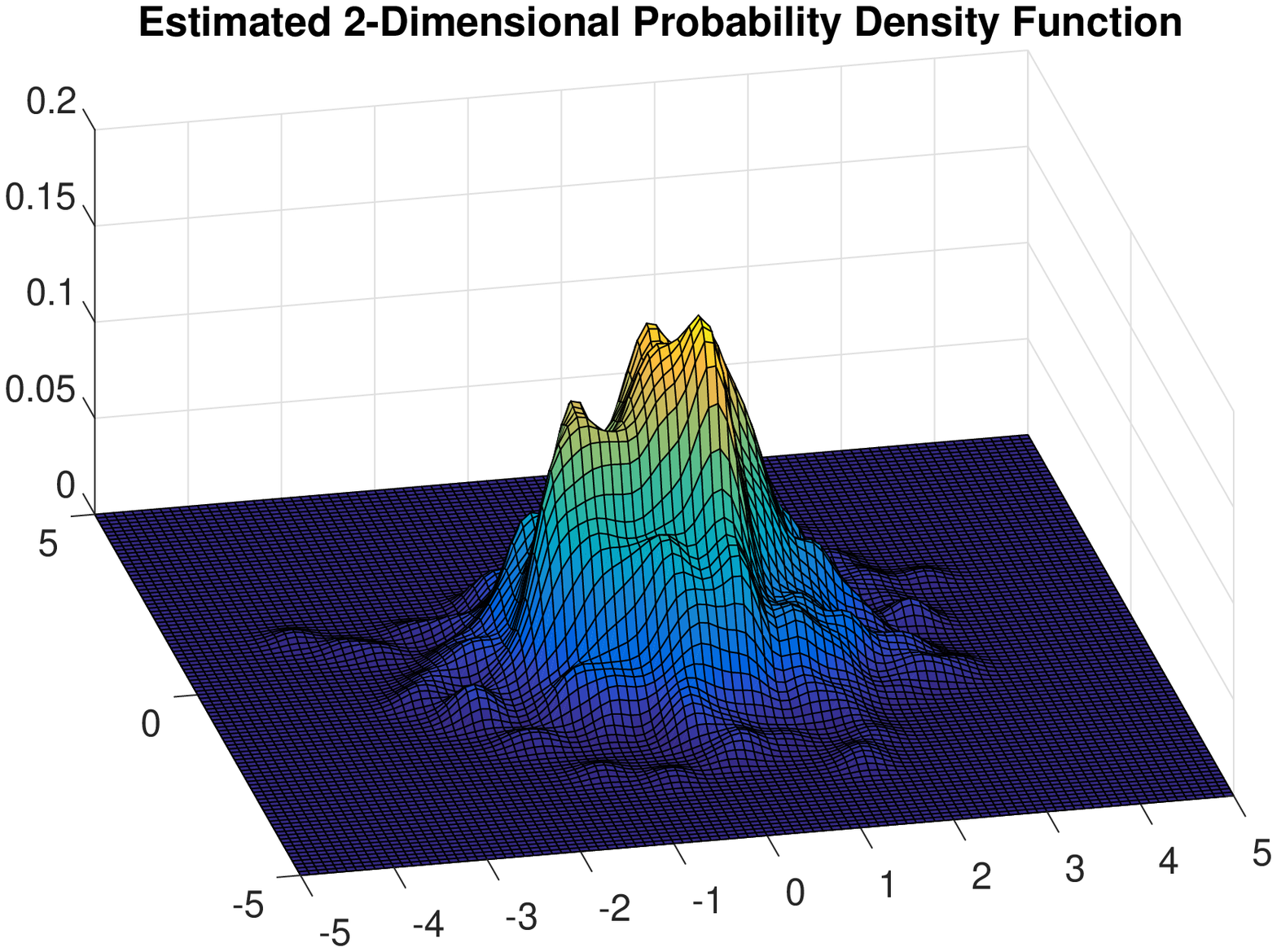}}\caption{Examples of density function estimates. \textit{(a)}: Actual and estimated probability density functions for the one-dimensional standard normal distribution with n=1000. \textit{(b)}: Estimated probability density function for the two-dimensional standard normal distribution with n=1000. Notice that the estimate appears to be more accurate for the one-dimensional case.}\label{fig:fig1}\end{figure}

The purpose of this paper is to demonstrate a method to reduce the error produced by using Kernel Density Estimation. The Curse of Dimensionality states that non-parametric methods such as this one will not be able to achieve an error of $O(n^{-1/2})$ for a large value of $d$, and although our new method does not quite reach this value, it does improve significantly on the value achieved by this base use of KDE. In the next section, we present the first step in this direction.

\section{Using the Richardson Extrapolation}
\label{sec:rich}
We now employ a method known as Richardson Extrapolation to decrease the bias and thereby reduce the error.

We pick two different values for $h$, which we call $h_1$ and $h_2$. Our estimated probability density function is now \[\hat{p}_{n,2}(x) = c_1\hat{p}_{n,h_1}(x)+c_2\hat{p}_{n,h_2}(x),\] where the subscript $2$ in $\hat{p}_{n,2}(x)$ indicates there are $2$ values for $h$, and $c_1$ and $c_2$ are constant weights. We will use this new estimate to decrease the value of the error.

\begin{theorem}\label{thm:newbmag}
  Using two values of $h$ with the estimator $\hat{p}_{n,2}(x)$, the bias is reduced to $O(h^4)$.
\end{theorem}
\begin{proof}
We first need to assume that our distribution $p$ is now four times differentiable. Then, recall from \cref{sec:bias} that \[\mathbb{E}[\hat{p}_{n,1}(x)] = p(x)+bias, \ \ \mbox{where}\ \  bias=O(h^2)=\frac{lh^2}{2}+O(h^4).\] Now, thanks to the linearity of expectation value, we have
\begin{align*}
\mathbb{E}[\hat{p}_{n,2}(x)] &= c_1\mathbb{E}[\hat{p}_{n,h_1}(x)]+c_2\mathbb{E}[\hat{p}_{n,h_2}(x)]\\
&=c_1(p(x)+\frac{lh_1^2}{2}+O(h_1^4))+c_2(p(x)+\frac{lh_2^2}{2}+O(h_2^4))\\
&=(c_1+c_2)p(x)+\frac{l}{2}(c_1h_1^2+c_2h_2^2)+O(h_1^4)+O(h_2^4).
\end{align*}

We now choose $c_1$ and $c_2$ so that the following matrix equation is true. The goal of this is to make $\mathbb{E}[\hat{p}_{n,2}(x)]$ be equal to p(x) with a bias of only $O(h^4)$. Note that because of the second part of this equation, one of $c_1$ or $c_2$ must be negative.
\[\begin{bmatrix} 1 & 1 \\ h_1^2 & h_2^2 \end{bmatrix}\begin{bmatrix} c_1 \\ c_2 \end{bmatrix} = \begin{bmatrix} 1 \\ 0 \end{bmatrix}.\] 
We can solve this for $c_1$ and $c_2$ whenever $h_1 \neq h_2$ since the determinant is then non-zero. After doing so, the summed $h^2$ terms of the biases cancel, and we get  \[\mathbb{E}[\hat{p}_{n,2}(x)] = p(x)+O(h_1^4)+O(h_2^4).\] 

Without loss of generality, we assume that $h_1$ is always less than $h_2$ so that we get \[\mathbb{E}[\hat{p}_{n,2}(x)] = p(x)+O(h_2^4).\]
\end{proof}

The bias has now been reduced from $O(h^2)$ to $O(h^4)$ - changed by two entire orders of magnitude.

Unlike the new bias, the following is true of the new variance.

\begin{theorem}
\label{thm:newvmag}
The order of magnitude of the variance is the same with this new estimator $\hat{p}_{n,2}(x)$ as with the original one: $O(1/(nh^d)).$ 
\end{theorem}

The proof of this is very similar to that of \cref{thm:varmag}, so we will not show it here, but the reader is welcome to find it in \cref{sec:vnewmag}.

To minimize the error, therefore, we set $h$ to be a constant times $n^{-1/(d+8)}$ and get an error value of \[\sigma^2+bias^2=O(h^8)=O(n^{-8/d+8}),\] which is visibly better than our previous $O(n^{-4/d+4})$, as shown in \cref{fig:fig2}. However, for large values of $d$, this is still a rather large error value. The next sections are dedicated to decreasing it further.

\begin{figure}[tbhp]\centering\subfloat[ ]{\label{fig:2a}\includegraphics[scale=.37]{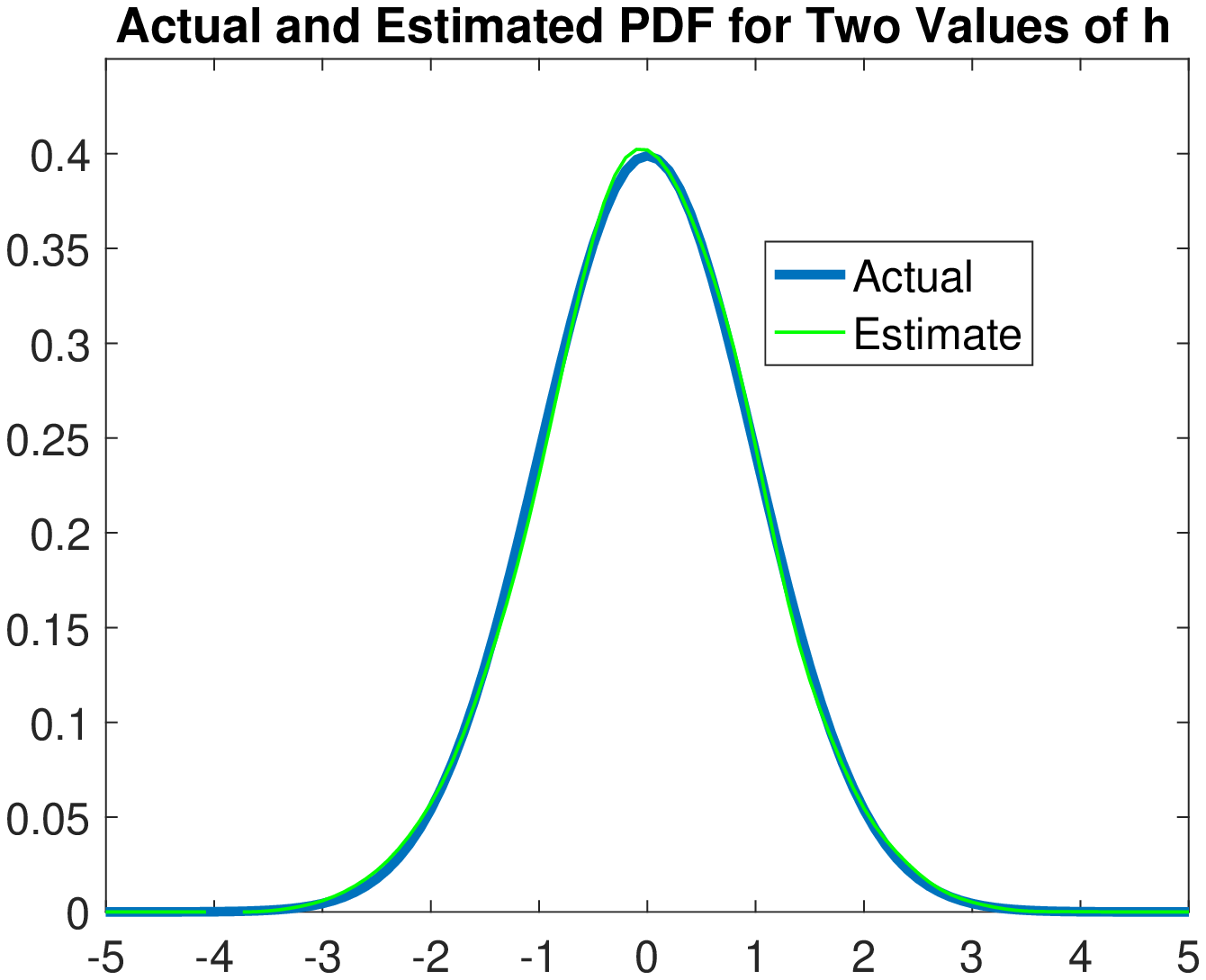}}\subfloat[ ]{\label{fig:2b}\includegraphics[scale=.33]{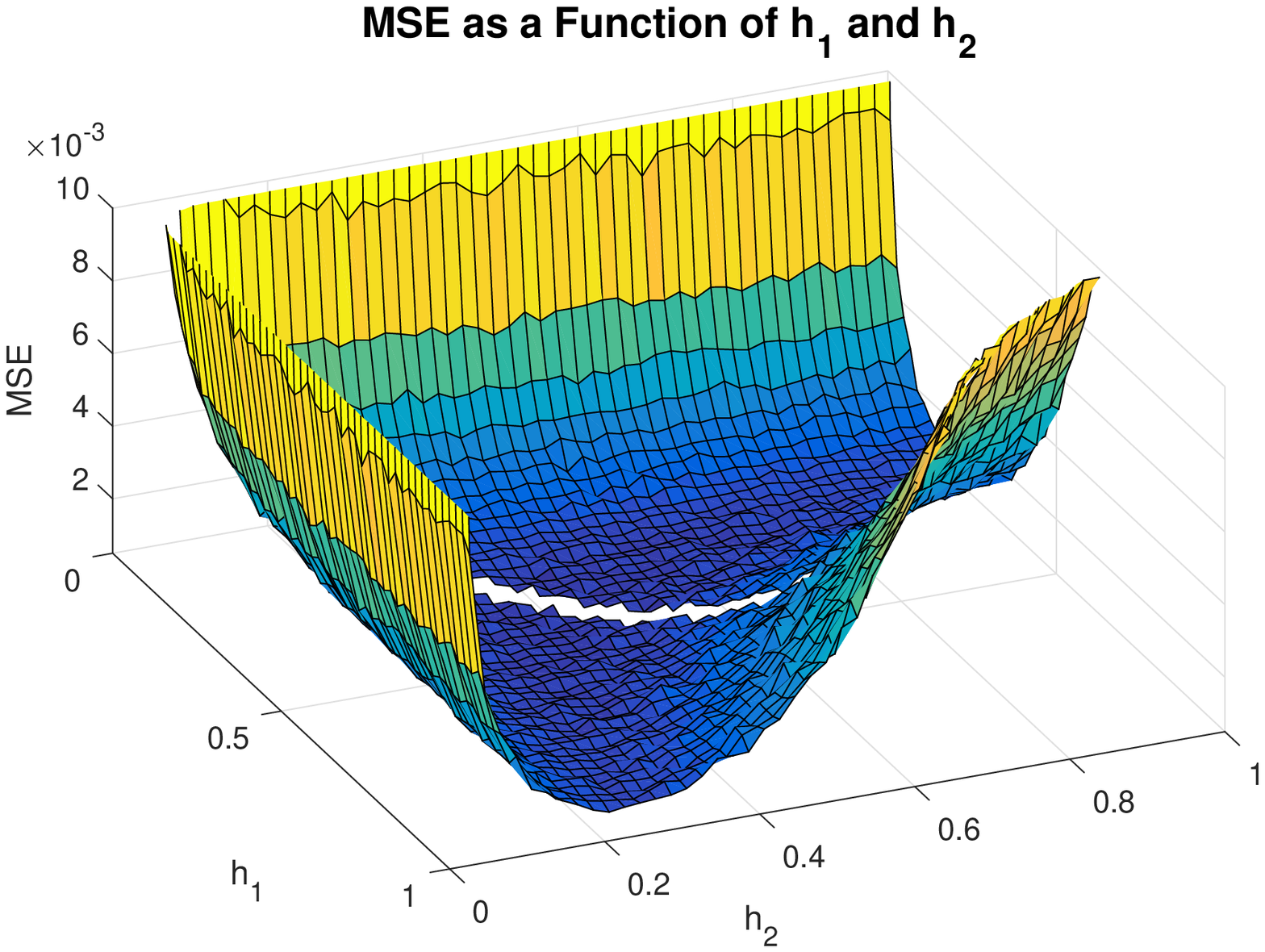}}\caption{A visual depiction of using two values of $h$. \textit{(a)}: Actual and estimated probability density functions for the one-dimensional standard normal distribution with n=1000 using $2$ values of $h$. You can see we have a much better estimate for the standard normal probability density function than in \cref{fig:1a}. \textit{(b)}: How to choose two values of $h$ to minimize the mean squared error. Note that the two values of $h$ cannot be the same, or we cannot find the values of $c_1$ and $c_2$. You can see how we should choose pairs of h values in the valley in order to minimize error.}\label{fig:fig2}\end{figure}

\section{Expanding on the Richardson Extrapolation}
\label{sec:lagrange}
There is no particular reason to stop at only two values for $h$. We will instead use $r$ values, keeping with the notation of Schucany (1977) \cite{schucany1977improvement}. We assume that $p(x)$ has $2r$ derivatives. The value of our estimate is then \[\hat{p}_{n,r}(x)=c_1\hat{p}_{n,h_1}(x)+c_2\hat{p}_{n,h_2}+...+c_r\hat{p}_{n,h_r},\] or if we define a vector $\vec{c}$ containing all the weights and a vector $\vec{p}_n(x)$ containing each individual estimate, then \[\hat{p}_{n,r}(x)=\vec{c}^{\ T}\vec{p}_n(x).\]

Then, to find the values of weights, we solve the following matrix equation for $c_1$ through $c_r$ by taking the inverse of the left-most matrix $R$ and left-multiplying it on both sides:
\[\begin{bmatrix}1 & 1 & 1 & ...\\ h_1^2 & h_2^2 & h_3^2 & ... \\ h_1^4 & h_2^4 & h_3^4 & ...\\ ... & ... & ... & ...  \end{bmatrix}\begin{bmatrix}c_1 \\ c_2 \\ c_3 \\ ...\end{bmatrix} = \begin{bmatrix}1 \\ 0 \\ 0 \\ ...\end{bmatrix}.\]

Using this method, if we choose $4r=d$, it seems we could in theory get the bias to be so small that our error converges to a limiting value of $O(n^{-4r/(4r+d)})=O(n^{-1/2})$ even for very large values of $d$. Unfortunately, it doesn't work out so nicely. The first obstacle to doing this is that $R$ becomes ill-conditioned as $r$ gets large, meaning that as consecutive rows get closer to $0$, the rows become linearly dependent to the top row of all ones, numerically speaking. Instead, we must solve for $\vec{c}$ without taking the inverse of $R$. In \cref{sec:limitations}, we will discuss the second obstacle to the coveted $n^{-1/2}$ result, but here we show the way past the first obstacle.

We start by taking the transpose of the equation above.
\[\begin{bmatrix}c_1 & c_2 & c_3 & ...\end{bmatrix}\begin{bmatrix} 1 & h_1^2 & h_1^4 & ... \\ 1 & h_2^2 & h_2^4 & ... \\ 1 & h_3^2 & h_3^4 & ... \\ ... & ... & ... & ...\end{bmatrix} = \begin{bmatrix}1 & 0 & 0 & ...\end{bmatrix}\]

The $r \times r$ matrix $R^T$ in the middle of this equation is called a Vandermonde matrix, except each term is squared. We call the row vector on the right hand side of this equation $\vec{k}^{\ T}$. The solution for $c_1$ through $c_r$ is \[\vec{c}^{\ T} = \vec{k}^{\ T}(R^T)^{-1}.\] Then, since $\vec{k}^{\ T}$ is a row vector with $1$ as the first element followed by all $0$'s, we know that $\vec{c}^{\ T}$ is simply the first row of $(R^T)^{-1}$. However, our goal is to avoid computing the inverse of $R^T$, so we will work around this using using the method of Lagrange Interpolation. We write
\[R^T\begin{bmatrix}a_0 \\ a_1 \\ a_2 \\ ...\end{bmatrix} = \begin{bmatrix}P_1(h_1^2) \\ P_1(h_2^2) \\ P_1(h_3^2) \\ ...\end{bmatrix}, \mbox{ where } P_1(h^2) = a_0 + a_1h^2+a_2h^4+...+a_{r-1}(h^2)^{r-1}.\]
Rearranging terms, we get
\[\begin{bmatrix}a_0 \\ a_1 \\ a_2 \\ ...\end{bmatrix} = (R^T)^{-1}\begin{bmatrix}P_1(h_1^2) \\ P_1(h_2^2) \\ P_1(h_3^2) \\ ...\end{bmatrix}.\]

In order to extract the first column of $(R^T)^{-1}$, we pick a polynomial $P_1(h^2)$ such that $P_1(h_1^2)=1$ and $P_1(h^2)=0$ for all other $h$'s. Using Lagrange Interpolation, such a polynomial looks like this: \[P_1(h^2) = \frac{(h^2-h_2^2)(h^2-h_3^2)...(h^2-h_r^2)}{(h_1^2-h_2^2)(h_1^2-h_3^2)...(h_1^2-h_r^2)}.\] 

Fortunately, we do not have to expand this completely; since we only need to find the first row of $(R^T)^{-1}$, we only need to compute $a_0$, the constant term. Since the column vector of polynomials on the right is simply $1$ followed by $0$'s, $a_0$ is equal to the element in the first row and first column of $(R^T)^{-1}$, which in turn is equal to $c_1$. So, we calculate \[c_1 = ((R^T)^{-1})_{11} = a_0 = \frac{(-h_2^2)(-h_3^2)...(-h_r^2)}{(h_1^2-h_2^2)(h_1^2-h_3^2)...(h_1^2-h_r^2)}.\] 

Next, we obtain $c_2$ by finding the element of the first row and second column of $(R^T)^{-1}$. This is done by finding the value of $a_0$ for a polynomial $P_2(h)$ that has value $1$ for $h=h_2$ and $0$ for all other $h$'s. As expected following the pattern, this comes out as \[c_2 = ((R^T)^{-1})_{12} = a_0 = \frac{(-h_1^2)(-h_3^2)(-h_4^2)...(-h_r^2)}{(h_2^2-h_1^2)(h_2^2-h_3^2)(h_2^2-h_4^2)...(h_2^2-h_r^2)}.\] Thus we generalize: the formula for $c_i$ is \[c_i = \prod_{j=1,\ j\neq i}^r \left(\frac{-h_j^2}{h_i^2-h_j^2}\right).\]

\section{Limitations due to the Derivative}
\label{sec:limitations}

Using this method, we can extend the order of the Richardson Extrapolation without having to worry about inverting the matrix. However, it is still not convenient to use a very large value of $r$. This is because, although it seems that MSE has order $O(n^{-4r/(4r+d)})$, we are neglecting the ``constant'' term multiplied by the bias, which although is constant for fixed $r$ ends up blowing up enormously with increasing values of $r$. 

Recall from \cref{sec:bias} that for $r=1$ the bias is equal to $\frac{lh^2}{2}$, where $l$ is bounded above by the maximum possible value of the second derivative of $p(x)$. For the general case with $r$ values of $h$, a straightforward generalization of \cref{thm:newbmag} shows that the bias becomes \[\sum_{i=1}^r\left(\frac{l_rh_i^{2r}}{2}\right),\] where $l_r$ is the weighted integral of the $(2r)$\textsuperscript{th} derivative of $p(x)$ times $u^{2r}$. We can calculate that $l_r$ is bounded by the maximum value of the $(2r)$\textsuperscript{th} derivative of $p(x)$. We must sum over all the different $h$'s, but if one $h$ is even slightly bigger than the others, then raising it to the power of $2r$ makes it much more significant than the others and this can simplify to $l_rh^{2r}/2$. 

As it turns out, for the example with a Gaussian density function, where $p(\rho)=e^{-\rho^2/2}/(2\pi)^{d/2}$, the value of $l_r$ blows up enormously for increasing values of $r$ - so much so that despite the fact that $h^{2r}$ continues to become smaller, the bias can no longer decrease.

In order to find the true order of magnitude of the bias, then, we must find the order of magnitude of $l_r$. This could vary between different probability density functions, so we will look at the standard normal distribution as an example, where $p(\rho)=e^{-\rho^2/2}/(2\pi)^{d/2}$. The $(2r)$\textsuperscript{th} derivative of this function with respect to $\rho$ can be expressed as $e^{-\rho^2/2}/(2\pi)^{d/2}$ times a polynomial. The maximum possible absolute value of a derivative is therefore the absolute value of the constant term of this polynomial divided by $(2\pi)^{d/2}$. The absolute value of this constant term is found to be $(2r-1)!!$, the so-called double factorial or semifactorial of $2r-1$. For example, for $r=4$, the constant term is $7!! = 7*5*3*1 = 105$, so the value of $l_r$ is $105/(2\pi)^{d/2}$. Note that $d$ is fixed and we are therefore treating $(2\pi)^{d/2}$ as a constant.

The next step is to figure out how big $(2r-1)!!$ is. We start by relating $(2r-1)!!$ to $(2r)!!$. The ratio between the two values is a constant times $\sqrt{r}$. We will neglect constants for now because we are only trying to find the order of magnitude. \[(2r-1)!! = \frac{(2r)!!}{\sqrt{r}}.\]

Next, note that the terms being multiplied together in $(2r)!!$ correspond to two times each term being multiplied in $r!$, so we can establish that $(2r)!! = 2^r*r!$. Thus, \[(2r-1)!! = \frac{2^r*r!}{\sqrt{r}}.\] Next, we use Stirling's Approximation to estimate $r!$ as $\sqrt{r}(r/e)^r$. Our equation becomes \[(2r-1)!! = \frac{2^r*\sqrt{r}}{\sqrt{r}}\left(\frac{r}{e}\right)^r = \left(\frac{2r}{e}\right)^r.\] 

The order of the bias is therefore  \[bias = O\left(r^r\left(\frac{2}{e}\right)^rh^{2r}\right).\] We now calculate the bias-variance tradeoff as we did in \cref{sec:mse}: 
\begin{align*}
bias^2 &= \sigma^2\\
O\left(r^{2r}\left(\frac{2}{e}\right)^{2r}h^{4r}\right)&=O\left(\frac{1}{nh^d}\right)\\
h &= \frac{1}{n^{1/(4r+d)}}\left(\frac{e}{2r}\right)^{2r/(4r+d)}.
\end{align*}

This gives us a value for the best theoretical $h$ in terms of $n$, $d$, and $r$. In practice, we choose $r$ values of $h$ that are approximately equal to this theoretical best value, but not actually equal, since that would result in division by zero when finding $\vec{c}$. We now compute the best value of $r$ - that is, the best order of the Richardson Extrapolation - in terms of $n$ and $d$. The MSE is equal to the bias squared plus the variance, so we plug this value of $h$ into the variance: \[MSE = O\left(\frac{1}{nh^d}\right) = O\left(\frac{1}{n^{4r/(4r+d)}}\left(\frac{2r}{e}\right)^{2rd/(4r+d)}\right).\]

Now we take the derivative with respect to $r$ and set it equal to $0$. (Note: the $d$ in the left side of the following equation is the differential; in the right side, it is the dimension.) \[\frac{d\ MSE}{dr} = 2^{\frac{2dr}{4r+d}+1}e^{-\frac{2dr}{4r+d}}r^{\frac{2dr}{4r+d}}n^{\frac{-4r}{4r+d}}(4r+d)^{-2}d(-2\ln(n)+4r+d\ln(r)+d\ln(2)) = 0\] 

Since only the last term here could equal $0$, we have \[(-2\ln(n)+4r+d\ln(r)+d\ln(2)) = 0.\]
Solving this equation for $r$ gives us a formula for the best value of $r$ given $n$ data points. The solution turns out to be \[r=\frac{1}{4}d\ W\left(\frac{2n^{2/d}}{d}\right),\] where $W(x)$ is the product log function, also known as the Lambert-W function, that can be expressed as the inverse of the function $f(W)=We^W$.

We estimate that $4r/(4r+d)=1/2$ and $2rd/(4r+d)=d/4$. Plugging this back into the formula for MSE, we get \[MSE = O\left(\frac{1}{n^{1/2}}\left(\frac{dW\left(\frac{2n^{2/d}}{d}\right)}{2e}\right)^{d/4}\right) = O(n^{-1/2}(dW(\alpha))^{d/4}),\] where $\alpha = 2n^{2/d}/d$. This is similar to the coveted goal value of $O(n^{-1/2})$, but with a correction factor that gets larger with larger $d$ and larger $n$.

\section{Summary of Results}
\label{sec:results}
It is thus shown that we can achieve a mean squared error value of $O(n^{-1/2}(dW((2n^{2/d})/d))^{d/4}).$ This result for large values of $d$ may be confusing, so we consider the case with $d=1$, when we want $MSE=O(n^{-1})$. Then, we estimate that $4r/(4r+1)=1$ and $2r/(4r+1)=\frac{1}{2}.$ We also argue that $r = W(\alpha)/4 = W(2n^2)/4$ is approximately equal to $\ln(2)/4+\ln(n)/2).$ Plugging this back into the formula for MSE, we get \[MSE = O\left(\frac{1}{n}\left(\frac{\ln(n)}{e}\right)^{1/2}\right) = O\left(\frac{1}{n}\sqrt{\ln(n)}\right).\] This is similar to the coveted goal value of $O(1/n)$ but with a correction factor of $\sqrt{\ln(n)}$.

\section{Future Research}
\label{sec:future}
\begin{figure}[!b]\centering\includegraphics[scale=0.75]{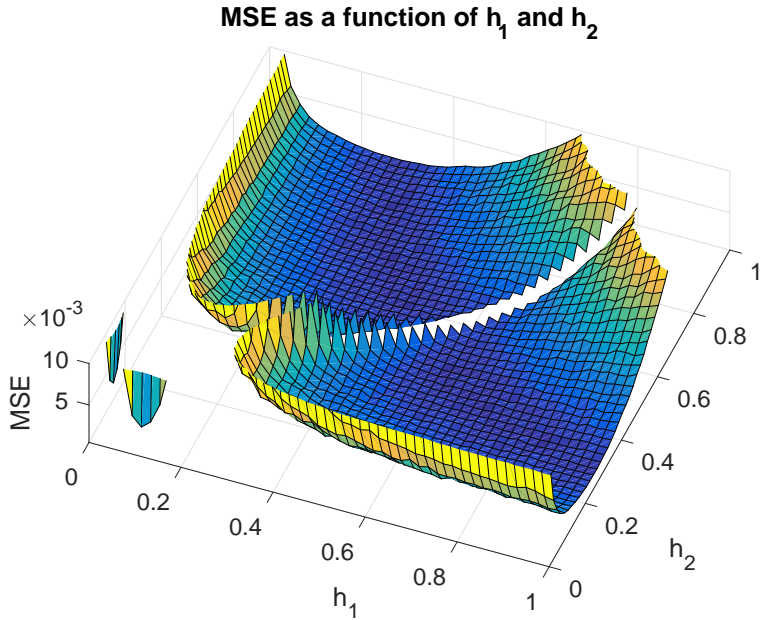}\caption{When choosing $c_1$ and $c_2$ to satisfy the constraint $\vec{c}^{\ T}V\vec{c} = \vec{c}^{\ T}B\vec{c}$, we choose pairs of $h$ values to minimize the mean squared error. The resulting MSEs are comparable to those using first order Richardson Extrapolation, shown in \cref{fig:2b}.}\label{fig:fig3}\end{figure}
Although this method improves the error value associated with the Kernel Density Estimator, it is very limited and certainly does not break the Curse of Dimensionality due to the correction factors associated with our final result. There is a variety of possible directions for further research, and here we present one of them. Instead of imposing many Richardson Extrapolation constraints in order to determine the values in $\vec{c}$, we use one fewer such constraint and instead impose the constraint $\vec{c}^{\ T}V\vec{c} = \vec{c}^{\ T}B\vec{c}$, where $V$ and $B$ are the variance and bias matrices, respectively:
\[V = \frac{1}{n}\begin{bmatrix}\frac{1}{h_1^d} & \frac{\sqrt{2}}{(h_1^2+h_2^2)^{d/2}} & ... \\ \frac{\sqrt{2}}{(h_1^2+h_2^2)^{d/2}} & \frac{1}{h_2^d} & ... \\ ... & ... & ... \end{bmatrix}\ \ \ \ \ B =\begin{bmatrix} h_1^4 & h_1^2h_2^2 & ... \\ h_1^2h_2^2 & h_2^4 & ... \\ ... & ... & ...\end{bmatrix}.\]

This could result in an even lower mean squared error. \cref{fig:fig3} shows the error when two values of $h$ are chosen and the values of $c_1$ and $c_2$ are computed using this constraint.\newpage

\appendix
\section{Proof of \cref{thm:newvmag}}
\label{sec:vnewmag}

The new calculation of the variance is the following, with similar simplification steps as with the case with a single value for $h$ in \cref{thm:varmag}. For conciseness, let \[K_h(x,a) = \frac{e^{-((x-a)/h)^2/2}}{(2\pi)^{d/2}h^d}.\]
\begin{align*}
\sigma^2 &= \mathbb{E}[(\hat{p}_{n,2}-\mathbb{E}[\hat{p}_{n,2}])^2]\\
&=\frac{1}{n^2}\mathbb{E}\left[\left(\sum_{i=1}^n (c_1(K_{h_1}(x,X_i)-\mathbb{E}[\hat{p}_n(x)])+c_2(K_{h_2}(x,X_i)-\mathbb{E}[\hat{p}_n(x)]))\right)^2\right]\\
&=\frac{1}{n}\mathbb{E}[ (c_1(K_{h_1}(x,X_i)-\mathbb{E}[\hat{p}_n(x)])+c_2(K_{h_2}(x,X_i)-\mathbb{E}[\hat{p}_n(x)]))^2]\\
&=\frac{1}{n}\int_{\mathbb{R}^d} (c_1(K_{h_1}(x,y)-p(x)+O(h_1^3))+c_2(K_{h_2}(x,y)-p(x)+O(h_2^3)))^2 p(y)\ dy\\
&=\frac{1}{n}\int_{\mathbb{R}^d} (c_1(K_{h_1}(x,y)-p(x)+O(h_1^3))+c_2(K_{h_2}(x,y)-p(x)+O(h_2^3)))^2 p(y)\ dy \\
&=\frac{1}{n}\int_{\mathbb{R}^d} (c_1^2(K_{h_1}(x,y))^2+c_2^2(K_{h_2}(x,y))^2+2c_1c_2(K_{h_1}(x,y))(K_{h_2}(x,y)))p(y)\ dy
\end{align*}

The first two terms of the integrand can be split into their own integrals, and the calculation becomes very similar to that of the variance in \cref{thm:varmag}. From those two terms, we get $O(\frac{1}{nh_1^d})+O(\frac{1}{nh_2^d})$. The last term requires some more calculation.
\begin{align*}
\frac{1}{n}\int_{\mathbb{R}^d} \left(2c_1c_2\frac{e^{-((x-y)/h_1)^2/2-((x-y)/h_2)^2/2}}{(2\pi)^d h_1^dh_2^d}\right)p(y)\ dy \\ = \frac{1}{n}\int_{\mathbb{R}^d} \left(2c_1c_2\frac{e^{-((x-y)^2/2)((h_1^2+h_2^2)/h_1^2h_2^2)}}{(2\pi)^d h_1^dh_2^d}\right)p(y)\ dy \\
= \frac{1}{n}\int_{\mathbb{R}^d} \left(2c_1c_2\frac{e^{-((x-y)/(h_1h_2/\sqrt{h_1^2+h_2^2})^2/2}}{(2\pi)^d h_1^dh_2^d}\right)p(y)\ dy \\
= \frac{1}{n}\int_{\mathbb{R}^d} \left(2c_1c_2\frac{e^{-((x-y)/v)^2/2}}{(2\pi)^d h_1^dh_2^d}\right)p(y)\ dy,
\end{align*}
where we substituted $v = \frac{h_1h_2}{\sqrt{h_1^2+h_2^2}}$. We now use the substitution \[u = \frac{y-x}{v}, \ \ du = det\left(\frac{du}{dy}\right)dy = \frac{1}{v^d}dy, \ \ y = uv+x \mbox{;    we then obtain}\]\[\frac{2v^dc_1c_2}{(2\pi)^d nh_1^dh_2^d}\int_{\mathbb{R}^d} e^{-u^2/2}p(uv+x)\ du = \frac{2\sqrt{2}p(x)c_1c_2}{(h_1^2+h_2^2)^{d/2}n(2\pi)^{d/2}}.\] 

We combine this result with the other two terms from the integral to get \[\sigma^2 = \frac{p(x)}{n(2\pi)^{d/2}}\left(\frac{c_1^2}{h_1^d}+\frac{c_2^2}{h_2^d}+\frac{2\sqrt{2}c_1c_2}{(h_1^2+h_2^2)^{d/2}}\right).\]  After all these calculations, we can see that the variance has not changed - it is still $O\left(\frac{1}{nh^d}\right)$.

\bibliographystyle{siamplain}
\bibliography{KDEbib}

\end{document}